\documentclass[a4paper,12pt]{amsart}

\numberwithin{equation}{section}
\allowdisplaybreaks

\newtheorem{thm}{Theorem}[section]
\newtheorem{prop}[thm]{Proposition}
\newtheorem{lem}[thm]{Lemma}
\newtheorem{cor}[thm]{Corollary}

\theoremstyle{definition}

\theoremstyle{remark}
\newtheorem{rmk}[thm]{Remark}

\newcommand{\CC}{\mathbb{C}}
\newcommand{\NN}{\mathbb{N}}
\newcommand{\RR}{\mathbb{R}}
\newcommand{\TT}{\mathbb{T}}
\newcommand{\ZZ}{\mathbb{Z}}

\def\F{\mathcal{F}}

\def\LL{\mathcal{L}}
\def\KK{\mathcal{K}}
\def\T{\mathcal{T}}
\def\O{\mathcal{O}}

\newcommand{\Ind}{\operatorname{Ind}}
\newcommand{\Tr}{\operatorname{Tr}}

\newcommand{\KMS}{\operatorname{KMS}}
\newcommand{\Aut}{\operatorname{Aut}}
\newcommand{\lsp}{\operatorname{span}}
\newcommand{\clsp}{\overline{\lsp}}

\title{\boldmath{KMS states on the $C^*$-algebras of finite graphs}}
\author[an Huef]{Astrid an Huef}
\author[Laca]{Marcelo Laca}
\author[Raeburn]{Iain Raeburn}
\author[Sims]{Aidan Sims}

\address{Astrid an Huef and Iain Raeburn\\ Department of Mathematics and Statistics\\University of Otago\\PO Box 56\\Dunedin 9054\\New Zealand}
\email{astrid@maths.otago.ac.nz, iraeburn@maths.otago.ac.nz}

\address{Marcelo Laca, Department of Mathematics and Statistics\\
University of Victoria\\
Victoria, BC V8W 3P4\\
Canada}
\email{laca@math.uvic.ca}

\address{Aidan  Sims\\ School of Mathematics and Applied Statistics\\
University of Wollongong\\NSW 2522\\Australia}
\email{asims@uow.edu.au}

\date{10 May 2011}

\thanks{This research has been supported by the University of Otago, the Marsden Fund of the Royal Society of New Zealand, the Natural Sciences and Engineering Research Council of Canada, and the Australian Research Council.}

\begin{document}

\begin{abstract}
We consider a finite directed graph $E$, and the gauge action on its Toeplitz-Cuntz-Krieger algebra, viewed as an action of $\RR$. For inverse temperatures
larger than a critical value $\beta_c$, we give an explicit construction of all the KMS$_\beta$ states. If the graph is strongly connected, then there is a unique
KMS$_{\beta_c}$ state, and this state factors through the quotient map onto $C^*(E)$. Our approach
is direct and relatively elementary.
\end{abstract}
\maketitle

Fix an integer $n\geq 2$, and consider the action  $\alpha$ of $\RR$ lifted from the gauge action of $\TT$ on the Cuntz algebra $\O_n$. Olesen and Pedersen \cite{OP} showed that $(\O_n,\alpha)$ has a unique KMS state, which occurs at inverse temperature $\ln n$. Enomoto, Fujii and Watatani \cite{EFW} extended this to the Cuntz-Krieger algebras $\O_A$. For an irreducible matrix $A$, they found that the unique KMS state has inverse temperature $\ln \rho(A)$, where $\rho(A)$ is the spectral radius of $A$, or equivalently the Perron-Frobenius eigenvalue of $A$.

There are now many generalisations of the Cuntz-Krieger algebras, and much is known about their KMS states. In particular, Exel and Laca \cite{EL} have conducted an extensive analysis of the KMS states on the Cuntz-Krieger algebras of infinite $\{0,1\}$-matrices. Their analysis is very general: they consider actions arising from embeddings of $\RR$ in the infinite torus $\TT^\infty$, and they study also the Toeplitz extension of the Cuntz-Krieger algebra, which has many more KMS states. So it was something of a surprise when Kajiwara and Watatani \cite{KW} pinpointed a gap in our knowledge of KMS states for the gauge action on the Cuntz-Krieger algebras of finite graphs: because Exel and Laca assumed that their matrices had no zero rows, their results do not apply to graphs with sources. Kajiwara and Watatani showed that the existence of sources makes a big difference (see Corollary~\ref{kwcor} below).

Here we describe the KMS states for the gauge action on the Toeplitz algebras of finite graphs. For a graph with vertex matrix $A$ and $\beta>\ln\rho(A)$, our main theorem gives an explicit isomorphism of the simplex of KMS$_\beta$ states on the Toeplitz algebra onto a simplex of codimension $1$ in $\CC^{E^0}$. Our methods are relatively elementary, and follow the general pattern of \cite{LR,LRR}.

After a brief review of background material, we begin in \S\ref{sec:id} by characterising KMS
states on graph algebras in terms of their behaviour on the usual spanning family. In
\S\ref{sec:Toe} we prove our main theorem about the KMS states on the Toeplitz algebra of a graph.
In \S\ref{sec:CK} we consider a strongly connected graph $E$, and prove that there is a unique KMS$_{\ln\rho(A)}$ state on $\T C^*(E)$, which factors through the KMS$_{\ln\rho(A)}$ state of $C^*(E)$ from \cite{EFW}.  In \S\ref{sec:ground}, we describe the ground and KMS$_\infty$
states. In Section~\ref{sec:connections}, we relate our results to those in \cite{EL}\ and
\cite{KW}. We also discuss how they relate to the powerful machines of Laca-Neshveyev \cite{LN} for
studying KMS states on Cuntz-Pimsner algebras and of Renault-Neshveyev \cite{Ren, N} for groupoid
algebras. We finish with an appendix on the possible values of the spectral radius for vertex
matrices of finite graphs.

\section{Background}\label{sec:back}
We use the conventions of \cite{R} for directed graphs $E$  and their $C^*$-algebras $C^*(E)$. We also borrow the convention from the higher-rank graph literature in which we write, for example, $E^*v$ for $\{\mu\in E^*:s(\mu)=v\}$, and $vE^1w$ for $\{e\in E^1:r(e)=v,\ s(e)=w\}$.

Suppose that $E$ is a directed graph. A \emph{Toeplitz-Cuntz-Krieger $E$-family $(P,S)$} consists of mutually orthogonal projections $\{P_v:v\in E^0\}$ and partial isometries $\{S_e:e\in E^1\}$ such that $S^*_eS_e=P_{s(e)}$ and
\begin{equation}\label{TCKcond}
P_v\geq \sum_{e\in F}S_eS_e^*\quad\text{for every finite subset $F$ of $vE^1$.}
\end{equation}
The definition used in \cite{FR} and \cite{R} requires that the partial isometries $\{S_e:e\in
E^1\}$ have mutually orthogonal ranges, but it turns out that this orthogonality follows from the
other relations. To see this we need a simple lemma.

\begin{lem}\label{critperp}
Suppose that $P$ and $Q$ are projections on a Hilbert space $H$ and $\|P+Q\|\leq 1$. Then $P$ and $Q$ have orthogonal ranges.
\end{lem}

\begin{proof}
Take $h\in PH$. Then
\[
\|h\|^2\geq \|Ph+Qh\|^2=(h+Qh\,|\,h+Qh)=\|h\|^2+3\|Qh\|^2, 
\]
and $\|Qh\|=0$. So $Qh=0$ for all $h\in PH$, and $QH$ is orthogonal to $PH$.
\end{proof}

\begin{cor}\label{oldTCKok}
Suppose that $\{P_v:v\in E^0\}$ are mutually orthogonal projections and $\{S_e:e\in E^1\}$ are partial isometries satisfying  \eqref{TCKcond}.
Then the projections $\{S_eS_e^*:e\in E^1\}$ are mutually orthogonal.
\end{cor}

\begin{proof}
Suppose first that $r(e)=r(f)=v$, say. Then \eqref{TCKcond} with $F=\{e,f\}$ shows that $P_v\geq S_eS_e^*+S_fS_f^*$. Since $T\geq S\geq 0$ implies $\|T\|\geq \|S\|$, we have $1=\|P_v\|\geq \|S_eS_e^*+S_fS_f^*\|$, and Lemma~\ref{critperp} implies that $S_eS_e^*$ and $S_fS_f^*$ are mutually orthogonal. On the other hand, if $r(e)\not=r(f)$, then applying \eqref{TCKcond} to singleton sets gives $S_eS_e^*\leq P_{r(e)}$ and $S_fS_f^*\leq P_{r(f)}$, and since $P_{r(e)}$ and $P_{r(f)}$ are mutually orthogonal, so are $S_eS_e^*$ and $S_fS_f^*$.
\end{proof}

The Toeplitz algebra $\T C^*(E)$ is generated by a universal   Toeplitz-Cuntz-Krieger family $(p,s)$. The existence of $\T C^*(E)$ was established in \cite[Theorem~4.1]{FR}, which says that the Toeplitz algebra $\T(X)$ of the associated graph correspondence $X$ has the required property. Corollary~\ref{oldTCKok} implies that $s_e^*s_f=\delta_{e,f}p_{s(e)}$, and then the usual argument for graph algebras (as in \cite[Corollary~1.15]{R}, for example) gives the product formula
\begin{equation}\label{prodform}
(s_\mu s_\nu^*)(s_\alpha s_\beta^*)=\begin{cases}
s_{\mu\alpha'}s_\beta^* & \text{if $\alpha=\nu\alpha'$}\\
s_\mu s_{\beta\nu'}^* &\text{if $\nu=\alpha\nu'$}\\
0&\text{otherwise.}\end{cases}
\end{equation}
From this formula, further standard arguments give
\[
\T C^*(E)=\clsp\{s_\mu s_\nu^*:\mu,\nu\in E^*,\ s(\mu)=s(\nu)\}.
\]
The Toeplitz algebra $\T C^*(E)$ carries a gauge action $\gamma$ of $\TT$ satisfying $\gamma_z(s_\mu s_\nu^*)=z^{|\mu|-|\nu|}s_\mu s_\nu^*$, and a dynamics $\alpha:\RR\to \Aut \T C^*(E)$ which is lifted from $\gamma$ via the map $t\mapsto e^{it}$. Since the quotient map $q$ of $\T C^*(E)$ onto $C^*(E)$ is gauge-invariant, we write $\alpha$ also for the corresponding action on the graph algebra $C^*(E)$. 

For every $\mu,\nu\in E^*$, the function $t\mapsto \alpha_t(s_\mu s_\nu^*)=e^{it(|\mu|-|\nu|)}s_\mu s_\nu^*$ on $\RR$ extends to an analytic function on all of $\CC$. Since these elements span a dense subspace of $\T C^*(E)$, it follows from \cite[Proposition~8.12.3]{P} that a state $\phi$ of $\T C^*(E)$ is a KMS$_\beta$ state of $(\T C^*(E),\alpha)$ for some $\beta\in \RR\setminus\{0\}$ if and only if
\begin{equation}\label{KMScond}
\phi((s_\mu s_\nu^*)(s_\sigma s_\tau^*))=\phi((s_\sigma s_\tau^*)\alpha_{i\beta}(s_\mu s_\nu^*))=e^{-\beta(|\mu|-|\nu|)}\phi((s_\sigma s_\tau^*)(s_\mu s_\nu^*))
\end{equation}
for all $\mu, \nu, \sigma, \tau\in E^*$. The KMS$_\beta$ states of $(C^*(E),\alpha)$ come from the KMS$_\beta$ states of $(\T C^*(E),\alpha)$ which factor through $q$.

As in \cite{P}, our KMS$_0$ states are the \emph{invariant} traces (as opposed to \emph{all} the traces, as in \cite{BR}). As in \cite{CM}, we distinguish between KMS$_\infty$ states, which are weak* limits of sequences of KMS$_{\beta_n}$ states as $\beta_n\to\infty$, and ground states, which are states $\phi$ such that the functions $\phi_{a,b}:z\mapsto \phi(a\gamma_z(b))$ are bounded on the upper-half plane for $a,b\in \{s_\mu s_\nu^*:\mu,\nu\in E^*\}$. In the older literature, such as \cite{BR} or \cite{P}, the KMS$_\infty$ states are defined to be the ground states. For general dynamical systems, every KMS$_\infty$ state is a ground state (by \cite[Theorem 5.3.23]{BR}), but not every ground state need be a KMS$_\infty$ state (as happens in \cite{LR} and \cite{LRR}, for example.)

\section{Characterising KMS states}\label{sec:id}

\begin{prop}\label{idKMSbeta}
Let $E$ be a finite directed graph, and let $A\in M_{E^0}(\NN)$ be the vertex matrix with entries
$A(v,w)=|vE^1w|$. Let $\gamma:\TT\to\Aut \T C^*(E)$ be the gauge action, and define
$\alpha:\RR\to\Aut\T C^*(E)$ by $\alpha_t=\gamma_{e^{it}}$. Let $\beta\in \RR$.
\begin{enumerate}
\item\label{prea}
A state $\phi$ of $\T C^*(E)$ is a KMS$_\beta$ state of $(\T C^*(E),\alpha)$ if and only if
\begin{equation}\label{charKMSbeta}
\phi(s_\mu s_\nu^*)=\delta_{\mu,\nu}e^{-\beta|\mu|}\phi(p_{s(\mu)})\quad\text{for all $\mu,\nu\in E^*$.}
\end{equation}
\item\label{charground} A state $\phi$ of $\T C^*(E)$ is a ground state of $(\T C^*(E),\alpha)$ if and only if
\begin{equation}\label{eqcharground}
\phi(s_\mu s_\nu^*)=0\quad\text{ whenever $|\mu|>0$ or $|\nu|>0$.}
\end{equation}
\item\label{a} Suppose that $\phi$ is a KMS${}_\beta$ state of $(\T C^*(E), \alpha)$, and
    define $m^\phi=(m^\phi_v)\in [0,\infty)^{E^0}$ by $m^\phi_v=\phi(p_v)$. Then $m^\phi$ is a
    probability measure on $E^0$ satisfying the \emph{subinvariance relation} $Am^\phi\leq
    e^\beta m^\phi$.
\item\label{whenCK} A KMS${}_\beta$ state $\phi$ of $(\T C^*(E), \alpha)$ factors through
    $C^*(E)$ if and only if $(Am^\phi)_v = e^\beta m^\phi_v$ whenever $v$ is not a source.
\end{enumerate}
\end{prop}

For part~\eqref{whenCK}, we need a lemma. Recall that the graph algebra $C^*(E)$ is the quotient of
$\T C^*(E)$ by the ideal $J$ generated by
\begin{equation}\label{defP}
P := \Big\{p_v-\sum_{f\in vE^1}s_fs_f^*:v\in E^0\text{ and }vE^1 \not= \emptyset\Big\}.
\end{equation}
We need to know that a state $\phi$ factors through $C^*(E)$ if and only if it vanishes on the
generators of $J$ (which is not obvious because $\phi$ is not a homomorphism). We have adapted the
following lemma from \cite[Lemma~10.3]{LR}, and have tried to phrase the new version so that it might be useful in other computations of KMS states. Notice that the sets $P$ defined in \eqref{defP} and
$\mathcal{F}:=\{s_\mu s_\nu^*\}$ in $\T C^*(E)$ have the properties required in
Lemma~\ref{checkprojOK}.

\begin{lem}\label{checkprojOK}
Suppose $(A,\RR,\alpha)$ is a dynamical system, and $J$ is an ideal in $A$ generated by a set $P$
of projections which are fixed by $\alpha$. Suppose that there is a family $\mathcal{F}$ of
analytic elements such that $\lsp\mathcal{F}$ is dense in $A$, and such that for each $a\in
\mathcal{F}$, there is a scalar-valued analytic function $f_a$ satisfying $\alpha_z(a)=f_a(z)a$. If
$\phi$ is a KMS$_{\beta}$ state of $(A,\alpha)$ and $\phi(p)=0$ for all $p\in P$, then $\phi$
factors through a state of $A/J$.
\end{lem}
\begin{proof}
It suffices to prove that $\phi(apb)=0$ for all $a,b\in A$ and $p\in P$. Let $p\in P$. We have
\[
0\leq \phi(paa^*p)\leq \phi(p\|a\|^2p)=\|a\|^2\phi(p)=0,
\]
and hence $\phi$ vanishes on $pAp$. Now fix $a,b\in\mathcal{F}$. Since
$\alpha_{t}(ap)=\alpha_t(a)p$ for $t\in \RR$, the element $ap$ is analytic with
$\alpha_z(ap)=\alpha_z(a)p=f_a(z)ap$.  Thus the KMS$_{\beta}$ condition gives
\[
\phi(apb)=\phi((ap)(pb))=\phi(pb\alpha_{i\beta}(ap))=f_a(i\beta)\phi(pbap)=0,
\]
and this extends to arbitrary $a$ and $b$ in $A$ by linearity and continuity of $\phi$. Now we use the linearity and continuity of $\phi$ again to see that $\phi$ vanishes on $J$.
\end{proof}

\begin{proof}[Proof of Proposition~\ref{idKMSbeta}]
\eqref{prea} Suppose that $\phi$ is a KMS$_\beta$ state. If $\beta\not=0$, then \cite[Proposition~8.12.4]{P} implies that $\phi$ is invariant for $\alpha$ and $\gamma$. This is also true for $\beta=0$ by our convention that the KMS$_0$ states are invariant traces. For $|\mu|\not=|\nu|$, invariance gives
\[
\phi(s_\mu s_\nu^*)=\int_{\TT}\phi(\gamma_z(s_\mu s_\nu^*))\,dz=\Big(\int_{\TT}z^{|\mu|-|\nu|}\,dz\Big) \phi(s_\mu s_\nu^*)=0.
\]
For $|\mu|=|\nu|$ the product formula \eqref{prodform} gives $s_\nu^*s_\mu=\delta_{\nu,\mu}p_{s(\mu)}$, and the KMS condition gives
\[
\phi(s_\mu s_\nu^*)=\phi(s_\nu^*\alpha_{i\beta}(s_\mu))=e^{-\beta|\mu|}\phi(s_\nu^*s_\mu)=\delta_{\mu,\nu}e^{-\beta|\mu|}\phi(p_{s(\mu)}),
\]
so $\phi$ satisfies \eqref{charKMSbeta}.

Next suppose that $\phi$ satisfies \eqref{charKMSbeta}. To see that $\phi$ is a KMS$_\beta$ state, it suffices to check the KMS condition \eqref{KMScond}. (For $\beta=0$, we need also to observe that any state $\phi$ satisfying \eqref{charKMSbeta} is automatically invariant for the gauge action --- indeed, $\phi(s_\mu s_\nu^*)\not=0$ implies $\mu=\nu$, and then $\gamma_z(s_\mu s_\nu^*)=s_\mu s_\nu^*$.) So we consider a pair of spanning elements $s_\mu s^*_\nu$ and $s_\sigma s^*_\tau$ in $\T C^*(E)$. Computations using the product formula \eqref{prodform} give
\begin{align*}
\phi(s_\mu s^*_\nu s_\sigma s^*_\tau)
&=\begin{cases}\phi(s_{\mu\sigma'} s^*_\tau) &\text{if $\sigma = \nu\sigma'$} \\
\phi(s_\mu s^*_{\tau\nu'}) &\text{if $\nu = \sigma\nu'$} \\
0 &\text{otherwise}
\end{cases} \\
&=\begin{cases}
e^{-\beta|\tau|} \phi(p_{s(\tau)}) &\text{if $\sigma = \nu\sigma'$ and $\tau = \mu\sigma'$} \\
e^{-\beta|\mu|} \phi(p_{s(\mu)}) &\text{if $\nu = \sigma\nu'$ and $\mu = \tau\nu'$} \\
0 &\text{otherwise.}
\end{cases}
\end{align*}
Similarly,
\[
\phi(s_\sigma s^*_\tau s_\mu s^*_\nu)
= \begin{cases}
e^{-\beta|\nu|} \phi(p_{s(\nu)}) &\text{if $\mu = \tau\mu'$ and $\nu = \sigma\mu'$} \\
e^{-\beta|\sigma|} \phi(p_{s(\sigma)})&\text{if $\tau = \mu\tau'$ and $\sigma = \nu\tau'$} \\
0 &\text{otherwise,}
\end{cases}
\]
and so
\[
\phi(s_\sigma s^*_\tau \alpha_{i\beta}(s_\mu s^*_\nu))
= \begin{cases}
e^{-\beta(|\mu| - |\nu|)} e^{-\beta|\nu|} \phi(p_{s(\nu)}) &\text{if $\mu = \tau\mu'$ and $\nu = \sigma\mu'$} \\
e^{-\beta(|\mu| - |\nu|)} e^{-\beta|\sigma|} \phi(p_{s(\sigma)}) &\text{if $\tau = \mu\tau'$ and $\sigma = \nu\tau'$} \\
0 &\text{otherwise.}
    \end{cases}
\]
If  $\mu = \tau\mu'$ and $\nu = \sigma\mu'$, then $s(\mu)=s(\nu)$ and
\[
\phi(s_\mu s^*_\nu s_\sigma s^*_\tau) =e^{-\beta|\mu|}\phi(p_{s(\mu)})=e^{-\beta|\mu|}\phi(p_{s(\nu)})=\phi(s_\sigma s^*_\tau
\alpha_{i\beta}(s_\mu s^*_\nu)).
\]
If $\tau =
\mu\tau'$ and $\sigma = \nu\tau'$, then $s(\tau)=s(\sigma)$ and $|\mu| - |\nu| = |\mu\tau'| - |\nu\tau'| =
|\tau| - |\sigma|$, so
\[
\phi(s_\mu s^*_\nu s_\sigma s^*_\tau) = e^{-\beta|\tau|}\phi(p_{s(\tau)})=e^{-\beta(|\mu|-|\nu|+|\sigma|)}\phi(p_{s(\sigma)})=e^{-\beta(|\mu|-|\nu|)}\phi(s_\sigma s^*_\tau s_\mu s^*_\nu).
\]
Otherwise both
$\phi(s_\mu s^*_\nu s_\sigma s^*_\tau)$ and $\phi(s_\sigma s^*_\tau \alpha_{i\beta}(s_\mu
s^*_\nu))$ are $0$. Thus $\phi$ satisfies \eqref{KMScond}, and is a $\KMS_\beta$-state.

\eqref{charground} For every state $\phi$ and all $\mu$, $\nu$ we have
\[
|\phi(s_\mu \alpha_{a + ib}(s_\nu^*))|= |e^{-i(a + ib)|\nu|}\phi(s_\mu s^*_\nu)|
= e^{b|\nu|}|\phi(s_\mu s^*_\nu)|.
\]
If $\phi$ is a ground state and $\phi(s_\mu s^*_\nu)\not=0$, this is bounded for $b>0$, and hence $|\nu|=0$; since $\phi(s_\nu s^*_\mu) = \overline{\phi(s_\mu s^*_\nu)}$, symmetry implies that $|\mu| = 0$ also. On the other hand, if $\phi$ satisfies \eqref{eqcharground}, then $|\phi(s_\mu \alpha_{a + ib}(s_\nu^*))|$ is constant, and $\phi$ is a ground state.

\eqref{a} Each $m_v^\phi$ is non-negative because $\phi$ is a positive functional. To see that
$m^\phi$ is a probability measure on $E^0$, note that $\sum_{v\in E^0} p_v$ is the identity of $\T
C^*(E)$, and hence
\[
1=\phi(1)=\sum_{v\in E^0}\phi(p_v)=\sum_{v\in E^0}m_v^\phi.
\]

Suppose $v \in E^0$ is not a source. We have $\phi(p_v) \ge \sum_{f \in vE^1} \phi(s_f s^*_f)$, and
\begin{equation}\label{eq-factorthrough}
\begin{split}
\sum_{f\in vE^1}\phi(s_fs_f^*)
    &=\sum_{f\in vE^1}e^{-\beta}\phi(p_{s(f)}) =\sum_{f\in vE^1}e^{-\beta}m_{s(f)}^\phi\\
    &=e^{-\beta}\sum_{w\in E^0} A(v,w)m_w^\phi =e^{-\beta}(Am^\phi)_v.
\end{split}
\end{equation}
Hence $(Am^\phi)_v \le e^{\beta}\phi(p_v) = e^{\beta} m^\phi_v$.

Now suppose $v\in E^0$ is a source. Then $A(v,w)=0$ for all $w\in E^0$, and
\[
(Am^\phi)_v=\sum_{w\in E^0} A(v,w)m_w^\phi=0\leq e^{\beta}m_v^\phi.
\]
Thus $(Am^\phi)_v\leq e^\beta m_v^\phi$ for all $v$.

\eqref{whenCK} Choose $v \in E^0$ and suppose that $v$ is not a source. By Lemma~\ref{checkprojOK}
it suffices to check that $\phi\big(p_v - \sum_{f\in vE^1} s_f s^*_f\big)=0$ if and only if
$(Am^\phi)_v = e^\beta m^\phi_v$. For this we use \eqref{charKMSbeta} and~\eqref{eq-factorthrough}
to see that
\[
e^\beta \phi\Big(p_v - \sum_{f \in vE^1} s_f s^*_f\Big)
    = e^\beta \Big(\phi(p_v) - \sum_{f \in vE^1} e^{-\beta}\phi(p_{s(f)})\Big)
    = e^\beta m^\phi_v - (Am^\phi)_v.\qedhere
\]
\end{proof}

\section{KMS states at large inverse temperatures}\label{sec:Toe}

In this section we study the KMS$_\beta$ states of $(\T C^*(E),\alpha)$ for $\beta > \ln\rho(A)$. The import of
this condition is that the series $\sum_{n=0}^\infty e^{-\beta n}A^n$ converges in the operator
norm with sum $(I-e^{-\beta A})^{-1}$ (by, for example, \cite[\S VII.3.1]{DS}). We use this
observation several times in the proof of Theorem~\ref{mainthmk=1}.

\begin{thm}\label{mainthmk=1}
Let $E$ be a finite directed graph with vertex matrix $A\in M_{E^0}(\NN)$.  Let $\gamma:\TT\to\Aut \T C^*(E)$ be the gauge action and define $\alpha:\RR\to\Aut\T C^*(E)$ by $\alpha_t=\gamma_{e^{it}}$. Assume that $\beta>\ln\rho(A)$.
\begin{enumerate}
\item\label{b} For $v\in E^0$, the series $\sum_{\mu\in E^*v}e^{-\beta|\mu|}$ either converges or is finite, with sum $y_v\geq 1$. Set $y:=(y_v)\in
    [1,\infty)^{E^0}$, and consider $\epsilon\in [0,\infty)^{E^0}$. Then
    $m:=(I-e^{-\beta}A)^{-1}\epsilon$ is a probability measure on $E^0$ if and only if
    $\epsilon\cdot y=1$.
\item\label{c} Suppose $\epsilon\in [0,\infty)^{E^0}$ satisfies $\epsilon\cdot y=1$, and set $m:=(I-e^{-\beta}A)^{-1}\epsilon$. Then there is a KMS$_\beta$ state $\phi_\epsilon$ of $(\T C^*(E), \alpha)$ satisfying
\begin{equation}\label{charKMSep}
\phi_\epsilon(s_\mu s_\nu^*)=\delta_{\mu,\nu}e^{-\beta|\mu|}m_{s(\mu)}.
\end{equation}
\item\label{doh} The map $\epsilon\mapsto\phi_\epsilon$ is an affine isomorphism of
\[
\Sigma_\beta:=\{\epsilon\in [0,\infty)^{E^0}:\epsilon\cdot y=1\}
\]
onto the simplex of KMS${}_\beta$ states of $(\T C^*(E), \alpha)$. The inverse of this isomorphism takes a KMS$_\beta$ state $\phi$ to $(I-e^{-\beta}A)m^\phi$.
\end{enumerate}
\end{thm}

\begin{rmk}
Because $y_v\geq 1$, the extreme points of $\Sigma_\beta$ are the vectors
$\epsilon^u=(\epsilon^u_v)=(\delta_{u,v}y_u^{-1})$. Thus $\Sigma_\beta$ is a simplex of dimension
$|E^0|-1$, as predicted by \cite{EL} (see \S\ref{sec:EL}).
\end{rmk}

\begin{proof}[Proof of Theorem~\ref{mainthmk=1}\,\eqref{b}]
Let $v\in E^0$. Since $A^n(w,v)$ is the number of paths of length $n$ from $v$ to $w$, we have
\begin{equation}\label{calcy}
\sum_{\mu\in E^*v}e^{-\beta|\mu|}=\sum_{n=0}^\infty\sum_{\mu\in E^nv}e^{-\beta n}=\sum_{n=0}^\infty\sum_{w\in E^0}e^{-\beta n}|wE^nv|=\sum_{n=0}^\infty \sum_{w\in E^0} e^{-\beta n}A^n(w,v).
\end{equation}
(The sums in \eqref{calcy} are finite if $E^nv$ is empty for large $n$.) Since $\beta>\ln\rho(A)$, the series $\sum_{n=0}^\infty e^{-\beta n}A^n$ converges in the operator norm. Thus for every fixed $w\in E^0$ the series $\sum_{n=0}^\infty e^{-\beta n}A^n(w,v)$ converges, and hence the last sum in \eqref{calcy} converges. The sum is at least $1$ because all the terms are non-negative and $e^{-\beta 0}A^0(v,v)=1$.

The expansion $m=\sum_{n=0}^\infty e^{-\beta n}A^n\epsilon$ shows that $m\geq 0$, and we use the same expansion to compute
\begin{align}
m(E^0)&=\sum_{v\in E^0}m_v=\sum_{v\in E^0}((I-e^{-\beta}A)^{-1}\epsilon)_v\label{eq-m}\\
&=\sum_{v\in E^0}\Big(\Big(\sum_{n=0}^\infty e^{-\beta n}A^n\Big)\epsilon\Big)_v
=\sum_{v\in E^0}\sum_{n=0}^\infty \sum_{w\in E^0}e^{-\beta n}A^n(v,w)\epsilon_w\notag\\
&=\sum_{w\in E^0}\epsilon_w\Big(\sum_{v\in E^0}\sum_{n=0}^\infty e^{-\beta n}|vE^n w|\Big)
=\sum_{w\in E^0}\epsilon_w\Big( \sum_{\mu\in E^*w} e^{-\beta|\mu|} \Big)\notag\\
&=\epsilon\cdot y.\notag\qedhere
\end{align}
\end{proof}

\begin{proof}[Proof of Theorem~\ref{mainthmk=1}\,\eqref{c}]
We build our KMS${}_\beta$ states by representing $\T C^*(E)$ on $\ell^2(E^*)$. We write $h_\mu$ for the point mass at $\mu\in E^*$, and let $\{Q_v:v\in E^0\}$ and $\{T_e:e\in E^1\}$ be the partial isometries such that
\[Q_v h_\mu=\begin{cases} h_{\mu}&\text{if $v=r(\mu)$}\\
0&\text{otherwise} \end{cases}
\quad\text{and}\quad T_e h_\mu=\begin{cases} h_{e\mu}&\text{if $s(e)=r(\mu)$}\\
0&\text{otherwise.} \end{cases}
\]
Then $(Q, T)$ is a Toeplitz-Cuntz-Krieger family in $B(\ell^2(E^*))$, and there is a representation $\pi_{Q,T}:\T C^*(E)\to B(\ell^2(E^*))$ such that $\pi_{Q,T}(p_v)=Q_v$ and $\pi_{Q,T}(s_e)=T_e$ (in fact, it follows from \cite[Theorem~4.1]{FR} that $\pi_{Q,T}$ is faithful). For $\mu\in E^*$ we set
\begin{equation}\label{defDelta}
\Delta_\mu:=e^{-\beta|\mu|}\epsilon_{s(\mu)},
\end{equation}
and note that $\Delta_\mu\geq 0$.  We aim to define $\phi_\epsilon$ by
\begin{equation}\label{claimstate}
\phi_\epsilon(a)=\sum_{\mu\in E^*}\Delta_\mu(\pi_{Q,T}(a)h_\mu\,|\, h_\mu)\quad\text{for $a\in \T C^*(E)$.}
\end{equation}

To see that \eqref{claimstate} defines a state, we need to show that $\sum_{\mu\in E^*}\Delta_\mu=1$. For $v\in E^0$ we have
\begin{align}\label{calcsum}
\sum_{\mu \in v E^*} \Delta_\mu
&=\sum^\infty_{n=0} \sum_{\mu \in v E^n} e^{-\beta n}\epsilon_{s(\mu)}=\sum^\infty_{n=0} e^{-\beta n}\Big(\sum_{w\in E^0}\sum_{\mu\in vE^nw}\epsilon_{w}\Big)\\
&=\sum^\infty_{n=0} e^{-\beta n}\Big(\sum_{w\in E^0}A^n(v,w)\epsilon_{w}\Big)= \sum^\infty_{n=0} (e^{-\beta n}A^n\epsilon)_{v},\notag
\end{align}
which converges with sum $m_v=\big((1-e^{-\beta}A)^{-1}\epsilon\big)_v$ because $\beta>\ln\rho(A)$. We saw in part~\eqref{b} that $m$ is a probability measure, so $\sum_{\mu\in E^*}\Delta_\mu=\sum_{v\in E^0}m_v=1$. This implies, first, that the series in \eqref{claimstate} converges for all $a$, and hence defines a positive functional on $\T C^*(E)$, and, second, that $\phi_\epsilon(1)=1$, so that $\phi_\epsilon$ is a state.

To prove \eqref{charKMSep}, let $\lambda \in E^*$. Then
\[
(\pi_{Q,T}(s_\mu s^*_\nu) h_\lambda \, |\,  h_\lambda) = (T_\nu^* h_\lambda \, |\,  T^*_\mu h_\lambda)= \begin{cases}
1 &\text{ if $\lambda = \mu\lambda' = \nu\lambda'$} \\
0 &\text{ otherwise.}
\end{cases}
\]
Since $\mu\lambda' = \nu\lambda'$ forces $\mu = \nu$, we have $\phi_\epsilon(s_\mu s^*_\nu) = 0$ if
$\mu \not= \nu$. So suppose $\mu = \nu$. Then since~\eqref{calcsum} gives $\sum_{\mu \in vE^*}
\Delta_\mu = m_v$, we have
\begin{align*}
\phi_\epsilon(s_\mu s^*_\mu)
&= \sum_{\lambda \in E^*} \Delta_\lambda (T^*_\mu h_\lambda\, |\, T_\mu^*  h_\lambda)
= \sum_{\lambda=\mu\lambda'} e^{-\beta|\mu\lambda'|}\epsilon_{s(\lambda')}\\
&=e^{-\beta|\mu|}\sum_{\lambda'\in s(\mu)E^*}\Delta_{\lambda'}=e^{-\beta|\mu|}m_{s(\mu)}.
\end{align*}
Thus $\phi_\epsilon$ satisfies \eqref{charKMSep}. Since $\phi(p_v)=m_v$, $\phi_\epsilon$ satisfies \eqref{charKMSbeta}, and Proposition~\ref{idKMSbeta} implies that $\phi_\epsilon$ is a KMS$_\beta$ state.
\end{proof}

\begin{proof}[Proof of Theorem~\ref{mainthmk=1}\,\eqref{doh}]
To see that every KMS$_\beta$ state $\phi$ has the form $\phi_\epsilon$, apply Proposition~\ref{idKMSbeta}\,\eqref{a} to see that $m^\phi=(\phi(p_v))$ is a subinvariant probability measure, and take $\epsilon:=(I-e^{-\beta}A)m^\phi$. Then $m:=(I-e^{-\beta}A)^{-1}\epsilon=m^\phi$, and comparing \eqref{charKMSbeta} with \eqref{charKMSep} shows that $\phi=\phi_\epsilon$.

The formula \eqref{charKMSep} also shows that the map $F:\epsilon\mapsto \phi_\epsilon$ is injective, and that $F$ is weak* continuous from $\Sigma_\beta\subset \RR^{E^0}$ to the state space of $\T C^*(E)$. Thus $F$ is a homeomorphism of the compact space $\Sigma_\beta$ onto the simplex of KMS$_\beta$ states. The formulas \eqref{defDelta} and \eqref{claimstate} show that $F$ is affine, and the formula for the inverse follows from the proof of surjectivity.
\end{proof}

\begin{rmk}
In the proof of Theorem~\ref{mainthmk=1}\,\eqref{doh}, we observed that a  probability measure $m$ on $E^0$ is subinvariant exactly when $\epsilon:=(I-e^{-\beta}A)m$ is nonnegative. Thus parts \eqref{b} and \eqref{c} of Theorem~\ref{mainthmk=1} give a converse to Theorem~\ref{idKMSbeta}\,\eqref{a}: the map $\phi\mapsto m^\phi$ defined by $m^\phi_v=\phi(p_v)$ is an isomorphism of the simplex of KMS$_\beta$ states onto the simplex of subinvariant probability measures on $E^0$.
\end{rmk}

\begin{rmk}\label{rho=0}
Whether the sum in \eqref{calcy} is finite or infinite depends on the existence of cycles in $E$. If the index set $E^*v$ is infinite for some $v$, then there is a cycle in $E$, Lemma~\ref{valuesrho} implies that $\rho(A)\geq 1$, and hence Theorem~\ref{mainthmk=1} applies only to $\beta>\ln \rho(A)\geq 0$. The index set $E^*v$ is finite for all $v$ if and only if there are no cycles in $E$, in which case Lemma~\ref{valuesrho} implies that $\rho(A)=0$, and Theorem~\ref{mainthmk=1} applies to any real $\beta$. Thus when $E$ has no cycles, there is a $(|E^0|-1)$-dimensional simplex of KMS$_\beta$ states for all $\beta\in \RR$.
\end{rmk}

\section{KMS states at the critical inverse temperature}\label{sec:CK}

The critical inverse temperature is $\beta=\ln\rho(A)$. As in the previous section, we can prove the existence of KMS$_{\ln\rho(A)}$ states for arbitrary graphs.

\begin{prop}\label{lemcriticalgenA}
Suppose that $E$ is a finite directed graph with vertex matrix $A$. If $m$ is a probability measure on $E^0$ such that $Am\leq \rho(A)m$, then there is a KMS$_{\ln\rho(A)}$ state $\phi$ on $(\T C^*(E),\alpha)$ such that
\begin{equation}\label{limphin2}
\phi(s_\mu s^*_\nu) = \delta_{\mu,\nu} \rho(A)^{-|\mu|} m_{s(\mu)}.
\end{equation}
The state $\phi$ factors through a state of $C^*(E)$ if and only if $(Am)_v=\rho(A)m_v$ for every vertex $v$ which is not a source.
\end{prop}

\begin{proof}
Choose a sequence $\{\beta_n\}\subset (\ln\rho(A),\infty)$ such that $\beta_n\to\ln\rho(A)$. For
each $n$, the vector $m$ is a probability measure satisfying $Am\leq\rho(A)m \le e^{\beta_n}m$.
Hence $\epsilon_n:=(I-e^{-\beta_n}A)m \in [0,\infty)^{E^0}$ and so the vector $y$ of
Theorem~\ref{mainthmk=1}\,\eqref{b} with $\beta = \beta_n$ satisfies $\epsilon_n \cdot y =
1$. Applying Theorem~\ref{mainthmk=1}\,\eqref{c} gives a KMS$_{\beta_n}$ state $\phi_n$ satisfying
\begin{equation}\label{phin}
\phi_n(s_\mu s^*_\nu) = \delta_{\mu,\nu} e^{-\beta_n|\mu|}m_{s(\mu)}.
\end{equation}
Since the state space of $\T C^*(E)$ is weak* compact we may assume by passing to a subsequence
that the sequence $\{\phi_{n}\}$ converges to a state $\phi$. Letting $n\to\infty$ in \eqref{phin}
shows that $\phi$ satisfies~\eqref{limphin2}. Proposition~\ref{idKMSbeta}\,\eqref{prea} implies
that $\phi$ is a KMS${}_{\ln\rho(A)}$ state. (Or we could apply the general result about limits of
KMS states in \cite[Proposition~5.3.23]{BR}.) The last assertion follows from
Proposition~\ref{idKMSbeta}\,\eqref{whenCK}.
\end{proof}

\begin{cor}\label{corcriticalA}
Suppose that $E$ is a finite directed graph with vertex matrix $A$. Then $(\T C^*(E),\alpha)$ has a KMS$_{\ln(\rho(A)}$ state.
\end{cor}

\begin{proof}
Choose a decreasing sequence $\{\beta_n\}$ with $\beta_n\to \ln\rho(A)$. Then Theorem~\ref{mainthmk=1}\,(a) implies that there are probability measures $m^n$ on $E^0$ such that $Am^n\leq e^{\beta_n}m^n$. By passing to a subsequence, we may suppose that $m^n$ converges pointwise to a probability measure $m$, and then $Am\leq \rho(A)m$. So Proposition~\ref{lemcriticalgenA} gives a KMS$_{\ln\rho(A)}$ state. 
\end{proof}

To get uniqueness at the critical inverse
temperature $\beta=\ln\rho(A)$, we impose some restrictions on $E$. We say that $E$ is \emph{strongly
connected} if $vE^*w$ is nonempty for every $v,w\in E^0$; equivalently, if the vertex matrix $A$ is
irreducible in the sense of \cite[Chapter~1]{Seneta}. Thus \cite[Theorem~1.5]{Seneta} implies that
$\rho(A)$ is an eigenvalue of $A$ for which there is an eigenvector $x=(x_v)$ with $x_v>0$ for all
$v$. The eigenvector $x$ such that $\sum_{v\in E^0} x_v=1$ is the \emph{unimodular
Perron-Frobenius eigenvector} of $A$. The following is our most satisfying result on uniqueness, but later we will improve it by allowing sources (Corollary~\ref{kwcor}). Notice in particular that parts~\eqref{cor-critical-a} and \eqref{cor-critical-b} complete the description of the KMS states on $\T C^*(E)$ when $E$ is strongly connected.

\begin{thm}\label{critical-irreducible}
Let $E$ be a finite directed graph and suppose that $E$ is strongly connected. Let $\gamma:\TT\to\Aut \T C^*(E)$ be the gauge action and define $\alpha:\RR\to\Aut\T C^*(E)$ by $\alpha_t=\gamma_{e^{it}}$. Let $x$ be the unimodular Perron-Frobenius eigenvector of the vertex matrix $A$.
\begin{enumerate}
\item\label{cor-critical-a} The system $(\T C^*(E),\alpha)$ has a unique
    KMS${}_{\ln\rho(A)}$ state $\phi$. This state satisfies
\begin{equation}\label{limphin}
\phi(s_\mu s^*_\nu) = \delta_{\mu,\nu} \rho(A)^{-|\mu|} x_{s(\mu)},
\end{equation}
and factors through a KMS${}_{\ln\rho(A)}$ state $\bar\phi$ of $(C^*(E), \alpha)$.
\item\label{criticalconv} The state $\bar\phi$ is the only KMS state of $(C^*(E),\alpha)$.
\item\label{cor-critical-b} If $\beta<\ln\rho(A)$, then $(\T C^*(E),\alpha)$ has no KMS${}_\beta$ states.
\end{enumerate}
\end{thm}

\begin{proof}
\eqref{cor-critical-a} We proved existence of $\phi$ in Corollary~\ref{corcriticalA}, and $\phi$
factors through $C^*(E)$ because $Ax=\rho(A)x$. To
establish uniqueness, suppose that $\psi$ is a KMS$_{\ln\rho(A)}$ state. Then
Proposition~\ref{idKMSbeta}\,\eqref{a} says that $m^\psi=(\psi(p_v))$ is a probability measure
satisfying $Am^\psi\leq \rho(A)m^\psi$. Now the forward implication in the last assertion of
\cite[Theorem~1.6]{Seneta} implies that $m^\psi=x$, and together the formulas \eqref{charKMSbeta}
and \eqref{limphin} imply that $\psi=\phi$.

\eqref{criticalconv} Suppose that $\psi$ is a KMS state of $(C^*(E),\alpha)$,  with inverse temperature $\beta$, say. Then Proposition~\ref{idKMSbeta}\,\eqref{whenCK} implies that $Am^{\psi\circ q}=e^\beta
m^{\psi\circ q}$; since $A$ is irreducible, the backward implication in the last assertion of
\cite[Theorem~1.6]{Seneta} implies that $e^\beta=\rho(A)$. Now the uniqueness in part~\eqref{cor-critical-a} implies that $\psi\circ q=\phi=\bar\phi\circ q$, and $\psi=\bar\phi$.

\eqref{cor-critical-b} Suppose that $\phi$ is a KMS$_\beta$ state of $(\T C^*(E),\alpha)$. Then
Proposition~\ref{idKMSbeta}\,\eqref{a} implies that $m^\phi:=(\phi(p_v))$ satisfies $Am^\phi\leq
e^{\beta}m^\phi$. In other words, $m^\phi$ is subinvariant. Since $m^\phi\geq 0$ pointwise,
\cite[Theorem~1.6]{Seneta} implies that $e^{\beta}\geq \rho(A)$, or equivalently $\beta\geq\ln\rho(A)$.
\end{proof}

\begin{rmk}
The irreducibility of $A$ was crucial in the proof of uniqueness. A similar phenomenon occurs in \cite[Theorem~5.3]{LRR}, where an extra hypothesis (there, that a given integer matrix is a dilation matrix) is needed to ensure uniqueness at the critical inverse temperature.
\end{rmk}

\begin{rmk}\label{rho=02}
A strongly connected graph contains at least one cycle. If it contains more than one, then Lemma~\ref{valuesrho} implies that $\rho(A)>1$, and $C^*(E)$ has no KMS$_0$ states. However, if $E$ consists of a single cycle, then $\rho(A)=1$, and Theorem~\ref{critical-irreducible} implies that $C^*(E)$ has a unique KMS$_0$ state. We check that this is consistent with what we know about cycles. Suppose that $E$ is a cycle with $n$ edges. Then $C^*(E)$ is isomorphic to $C(\TT,M_n(\CC))$, and the gauge action acts transitively on the spectrum $\TT$ (see \cite[Example~2.14]{R}, for example). Thus there is a unique invariant measure, and integrating the usual normalised trace against this measure gives a unique invariant trace on $C(\TT,M_n(\CC))$.
\end{rmk}

\section{Ground states and KMS$_\infty$ states}\label{sec:ground}

\begin{prop}
Let $E$ be a finite directed graph, let $\gamma:\TT\to\Aut \T C^*(E)$ be the gauge action and
define $\alpha:\RR\to\Aut\T C^*(E)$ by $\alpha_t=\gamma_{e^{it}}$. Suppose that $\epsilon$ is a
probability measure on $E^0$. Then there is a KMS${}_\infty$ state $\phi_\epsilon$
satisfying
\begin{equation}\label{defphie}
\phi_\epsilon(s_\mu s_\nu^*)=
\begin{cases}
0&\text{unless $|\mu|=|\nu|=0$ and $\mu=\nu$}\\
\epsilon_v&\text{if $\mu=\nu=v\in E^0$.}
\end{cases}
\end{equation}
Every ground state of $(\T C^*(E),\alpha)$ is a KMS${}_\infty$ state, and the map $\epsilon\mapsto
\phi_\epsilon$ is an affine isomorphism of the simplex of probability measures on $E^0$ onto the
set of ground states of $(\T C^*(E),\alpha)$.
\end{prop}

\begin{proof}
Choose a sequence  $\beta_j\to \infty$ as $j\to \infty$ with each $\beta_j>\ln \rho(A)$. For each
$j$, define $(y^j_v)$ as in Theorem~\ref{mainthmk=1}\,\eqref{b} by $y^j_v=\sum_{\alpha\in
E^*v}e^{-\beta_j|\alpha|}$. Set $\epsilon^j_v:=\epsilon_v(y_v^j)^{-1}$, and let $\phi_j$ be the
KMS$_{\beta_j}$ state $\phi_{\epsilon^j}$ of $(\T C^*(E),\alpha)$ described in Theorem~\ref{mainthmk=1}\,\eqref{c}.
Since the state space is weak* compact we may assume that $\phi_j$ converges in the weak* topology
to a state $\phi_\epsilon$.

Set $m^j := (I - e^{-\beta_j}A)^{-1} \epsilon^j$, and take $\mu$, $\nu$ in $E^*$. Then
\[
\phi_j(s_\mu s^*_\nu) = \delta_{\mu,\nu}
e^{-\beta_j|\mu|} m^j_{s(\mu)}.
\]
This is always $0$ if $\mu\not=\nu$, so suppose that $\mu = \nu$. If $|\mu| > 0$, then $e^{-\beta_j
|\mu|} \to 0$, and hence $\phi_{j}(s_\mu s^*_\mu) \to 0$. So suppose that $\mu = \nu = v$ is a
vertex. An application of the dominated convergence theorem shows that $y^j_v\to 1$ as
$j\to\infty$. Hence $\epsilon^j_v =\epsilon_v(y_v^j)^{-1} \to \epsilon_v$. Since $(I -
e^{-\beta_j}A)^{-1} \to I$ in the operator norm, we have $ m^j_v\to\epsilon_v$, and hence
$\phi_j(p_v) \to \epsilon_v$. Since $\phi_j(p_v)\to \phi_{\epsilon}(p_v)$, we deduce that $\phi_\epsilon(p_v) = \epsilon_v$, and $\phi_\epsilon$
satisfies~\eqref{defphie}.

Since $\phi_\epsilon(s_\mu s_\nu^*)=0$ whenever $|\mu|>0$ or $|\nu|>0$, Proposition~\ref{idKMSbeta}\,\eqref{charground} implies that $\phi_\epsilon$ is a
ground state.
Now let $\phi$ be a ground state, and $\epsilon_v:=\phi(p_v)$.  Then $\epsilon$ is a probability measure on $E^0$, and $\phi=\phi_\epsilon$ because $\phi$ is determined by its values on the vertex projections (by Proposition~\ref{idKMSbeta}\,\eqref{charground} again). Thus $\epsilon\mapsto \phi_\epsilon$ maps the simplex of probability measures onto the ground states, and it is clearly affine and injective. Since each $\phi_\epsilon$ is by construction a KMS$_\infty$ state, it follows that every ground state is a KMS$_\infty$ state.
\end{proof}

\section{Connections with the literature}\label{sec:connections}

\subsection{Cuntz-Krieger algebras}\label{sec:EL}
When $E$ has no sources, we could in principle deduce Theorem~\ref{mainthmk=1} from the work of
Exel and Laca on the Cuntz-Krieger algebras of $\{0,1\}$-matrices. To do this, we apply
\cite[Theorem~18.4]{EL} to the edge matrix $B\in M_{E^1}(\NN)$ of $E$ defined by
\[
B(e,f):=\begin{cases}
1&\text{if $s(e)=r(f)$}\\
0&\text{otherwise.}
\end{cases}
\]
We need to assume that $E$ has no sources, because an edge $e$ such that $s(e)$ is a source would
give a row of zeroes in $B$, which is not allowed in \cite{EL}. Edges with the same range give
equal columns of $B$, so provided there are no sinks, the distinct columns of $B$ are in one-to-one
correspondence with the vertex set $E^0$, and the set $\Omega_e$ in \cite[Theorem~18.4]{EL} can be
identified with $E^0$. The condition $\epsilon\cdot y=1$ is phrased in \cite{EL} as
$Z(\beta,\epsilon)=1$, but we think it would be hard to dig our formula for $\phi_\epsilon$ out of
\cite{EL}.

\subsection{Graphs with sources}   Kajiwara and Watatani \cite{KW} have recently considered the KMS states on the graph algebras $C^*(E)$ of arbitrary finite graphs, and have shown in particular that sources give rise to KMS$_\beta$ states \cite[Theorem~4.4]{KW}. 

Recall that a subset $H$ of $E^0$ is \emph{saturated} if $s(vE^1)\subset H\Longrightarrow v\in H$. The \emph{saturation} of a hereditary set $S\subset E^0$ is the smallest saturated set $H$ containing $S$. The saturation is itself hereditary. For a saturated hereditary set $H$, the $C^*$-algebra $C^*(E\setminus H)$ of the graph in \eqref{CKsources} below is a quotient of $C^*(E)$ (see \cite[Theorem~4.9]{R}). 

\begin{cor}\label{kwcor} Let $E$ be a finite directed graph with vertex matrix $A$.  Let $\gamma:\TT\to\Aut C^*(E)$ be the gauge action and define $\alpha:\RR\to\Aut C^*(E)$ by $\alpha_t=\gamma_{e^{it}}$.
\begin{enumerate}
\item\label{noncrit} Assume that $\beta>\ln\rho(A)$, and that $\epsilon$ belongs to the simplex
    $\Sigma_\beta$ of Theorem~\ref{mainthmk=1}. Then $\phi_\epsilon$ factors through a state of
    $C^*(E)$ if and only if $\epsilon_v=0$ whenever $v$ is not a source.
\item\label{CKsources} Let $H$ be the saturation of the set of sources in $E$. Suppose that $E$ has no sinks and that the graph $E\setminus H:=(E^0\setminus H, E^1\setminus s^{-1}(H),r,s)$ is strongly connected. Let $x$ be the unimodular Perron-Frobenius eigenvector for the vertex matrix $A_{E\setminus H}$ of $E\setminus H$. Then there is a unique KMS$_{\ln \rho(A)}$ state $\phi$ of $(C^*(E),\alpha)$, and
\begin{equation}\label{phiwhensources}
\phi(s_\mu s_\nu^*)=\begin{cases}
\delta_{\mu,\nu}\rho(A)^{-|\mu|}x_{s(\mu)}&\text{if $s(\mu)\in E^0\setminus H$}\\
0&\text{if $s(\mu)\in H$.}
\end{cases}
\end{equation}
\end{enumerate}
\end{cor}

\begin{proof}
\eqref{noncrit} We set  $m:=(I-e^{-\beta}A)^{-1}\epsilon$.
Proposition~\ref{idKMSbeta}\,\eqref{whenCK} implies that $\phi_\epsilon$ factors through a state of
$C^*(E)$ if and only if $m_v = (e^{-\beta} Am)_v$ whenever $v$ is not a source. Since $\epsilon =
(I - e^{-\beta} A)m$, this is equivalent to $\epsilon_v = 0$ whenever $v$ is not a source.

\eqref{CKsources} Corollary~\ref{corcriticalA} implies that there is a KMS$_{\ln\rho(A)}$ state $\psi$ on $\T C^*(E)$. Then Proposition~\ref{idKMSbeta}\,\eqref{a} says that $(m^\psi_v):=(\psi(p_v))$ gives a probability measure $m^\psi$ satisfying the subinvariance relation $Am^\psi\leq \rho(A)m^\psi$. We will analyse this subinvariance relation by writing $A$ in a particular block form. 

To do this, let
$S$ be the set of sources in $E$. Recall, from \cite[Remark~3.1]{BHRS} for example, that we can
construct the saturation $H$ of the hereditary set $S$ inductively. We set $S_0:=S$, and for $k\geq
0$ set
\[
S_{k+1}=S_k\cup \{v\in E^0: s(vE^1)\subset S_k\}.
\]
Since $E$ is finite, we eventually have $S_{n+1}=S_n$, and then $H=S_n$. We now order $E^0$ by
listing first the vertices in $E^0\setminus H$, then those in $H\setminus S_{n-1}$, then those in
$S_{n-1}\setminus S_{n-2}$ and so on, finishing with the sources $S$. This ordering gives a block
decomposition
\[
A=\begin{pmatrix}A_{E\setminus H}&B\\0&A_H\end{pmatrix}.
\]
Since $w \in S_{k+1}$ and $A(w,v) \not= 0$ imply $v \in S_k$, the matrix $A_H$ is strictly upper
triangular, in the sense that all entries on or below the diagonal are $0$.

We write $m^\psi=(m^{E\setminus H},m^H)$ in block form. Subinvariance says that
\begin{align}
\label{offH} A_{E\setminus H}m^{E\setminus H}+Bm^H&\leq \rho(A)m^{E\setminus H}, \text{ and}\\
\label{onH} A_Hm^H&\leq \rho(A)m^H.
\end{align}
Since $A_H$ is strictly upper triangular, $\rho(A_{E\setminus H})=\rho(A)$. Thus, since $B$ and
$m^H$ are non-negative, \eqref{offH} implies that $A_{E\setminus H}m^{E\setminus H}\leq
\rho(A_{E\setminus H})m^{E\setminus H}$. Now the last assertion in \cite[Theorem~1.6]{Seneta}
implies that $A_{E\setminus H}m^{E\setminus H}=\rho(A_{E\setminus H})m^{E\setminus H}$, and
$m^{E\setminus H}$ is a Perron-Frobenius eigenvector for $A_{E\setminus H}$.

Since $A_{E\setminus H}m^{E\setminus H}=\rho(A_{E\setminus H})m^{E\setminus H}$, \eqref{offH}
implies that $Bm^H=0$. We next claim that $m^H=0$, or equivalently that $m^\psi_v=0$ for every
vertex $v\in H$. First we consider $v\in S_n\setminus S_{n-1}$. Since $E$ has no sinks, there is an
edge $e$ with $s(e)=v$. The range $r(e)$ cannot be in $H$, because otherwise $r(e)\in S_n$ but
$s(e)$ is not in $S_{n-1}$. Thus $B(r(e),v)>0$, and $(Bm^H)_{r(e)}=0$ implies $m^H_v=0$. Thus
$m^\psi_v=0$ for all $v\in S_n\setminus S_{n-1}$. Now we repeat the argument with $w\in
S_{n-1}\setminus S_{n-2}$, and find that $s(e)=w$ implies that $r(e)$ is in either $E\setminus H$
or $S_n\setminus S_{n-1}$; thus one of $B(r(e),w)m_w=0$ or $0\leq A_H(r(e),w)m^\psi_w\leq
\rho(A)m^\psi_{r(e)}$ forces $m^\psi_w=0$. A finite induction argument gives $m^\psi_v=0$ for all
$v\in H$, as claimed.

Since $m^\psi$ is a probability measure on $E^0$, we deduce that $m^{E\setminus H}$ is a probability measure too. Thus $m^{E\setminus H}$ is the unimodular Perron-Frobenius eigenvector $x$ for $A_{E\setminus H}$. Now the vector $m^\psi$ has block form $(x,0)$, and we have $Am^\psi=\rho(A)m^\psi$. Thus it follows from Proposition~\ref{idKMSbeta}\,\eqref{whenCK} that $\psi$ factors through a state $\phi$ of $C^*(E)$, and this is a KMS$_{\ln\rho(A)}$ state of $(C^*(E),\alpha)$. The formula in Proposition~\ref{idKMSbeta}\,\eqref{prea} implies that $\phi$ satisfies \eqref{phiwhensources}.

To see uniqueness, suppose $\phi'$ is a KMS$_{\ln\rho(A)}$ state of $(C^*(E),\alpha)$. Then we can run the argument of the preceding five paragraphs with $\psi=\phi'\circ q$, and deduce that $\phi'$ also satisfies \eqref{phiwhensources}. Thus $\phi'$ and $\phi$ agree on the spanning elements $s_\mu s_\nu^*$, and hence are equal.
\end{proof}

Corollary~\ref{kwcor}\,\eqref{noncrit} implies that for $\beta>\ln\rho(A)$ the simplex of KMS$_\beta$ states of
$(C^*(E),\alpha)$ has dimension $|S|$ where $S$ is the set of sources.
Corollary~\ref{kwcor}\,\eqref{noncrit} contains \cite[Theorem~4.4]{KW} because in both situations
considered there, the range of inverse temperatures is the same as ours. Our result applies to
arbitrary finite graphs, and not just ones whose vertex matrix has entries in $\{0,1\}$, as in
\cite{KW}. The uniqueness in part \eqref{CKsources} appears to be new.

\subsection{Cuntz-Pimsner algebras}
When we were proving Theorem~\ref{mainthmk=1} we were guided by the
machinery developed in \cite{LN}, and in particular by the construction in the proof of \cite[Theorem~2.1]{LN}. We believe that our direct approach will be more accessible to more readers, but putting our calculations in the context of \cite{LN} might provide an illuminating example for those interested in the general machine.

To apply the ideas of \cite{LN}, we view $\T C^*(E)$ as the Toeplitz algebra $\T(X)$ of the graph correspondence $X$ as in \cite[Theorem~4.1]{FR} and \cite[\S8]{R}. We consider only the Toeplitz algebra because there are competing definitions of the Cuntz-Pimsner algebra $\O_X$ in the literature, and when $E$ has sources the one used in \cite{LN} is not the one which gives $C^*(E)$\footnote{The ideal $I_X$ such that $\O_X=\T(X)/I_X$ in \cite{LN} is $\phi^{-1}(\KK(X))$ rather than $J_X:=\phi^{-1}(\KK(X))\cap (\ker\phi)^\perp$ (see the discussion in \cite[Chapter~8]{R} or \cite[\S4]{KW}). Kajiwara and Watatani deduce their results about KMS states on $C^*(E)$ for $E$ with sources from a version of \cite[Theorem~2.5]{LN} for algebras of the form $\T(X)/J_X$ \cite[Theorem~3.17]{KW}.}.

In the conventions of \cite{R}, $X$ has underlying set $C(E^1)$, with module actions given by
$(a\cdot x\cdot b)(e)=a(r(e))x(e)b(s(e))$ and inner product by $\langle
x,y\rangle(v)=\sum_{s(e)=v}\overline{x(e)}y(e)$. We can similarly realise the tensor powers
$X^{\otimes n}$ as bimodules with underlying set $C(E^n)$. The dynamics on $\T(X)$ is implemented
by the unitaries $U_t^n$ on $X^{\otimes n}=C(E^n)$ such that $U_t^n\delta_\mu=e^{itn}\delta_\mu$,
and hence the operator $\Gamma(e^{-\beta D})$ on the Fock module $\F(X)=\bigoplus_{n\geq
0}X^{\otimes n}$ considered in \cite{LN} is multiplication by $e^{-\beta n}$ on $X^{\otimes n}$.

A function $\epsilon\in \Sigma_\beta$ gives a trace $\tau_\epsilon: f\mapsto \sum_{v\in
E^0}f(v)\epsilon_v$ on $C(E^0)$. For each $n$, the identity operator on $C(E^n)=X^{\otimes n}$
satisfies $1_n=\sum_{\lambda\in E^n}\Theta_{\delta_\lambda,\delta_\lambda}$. Thus the induced trace
$F^n\tau_\epsilon$ on $A=C(E^0)\subset \LL(\F(X))$ constructed in \cite[Theorem~1.1]{LN} is given
by
\[
(F^n\tau_\epsilon)(\delta_w)=\sum_{\lambda\in E^n}\tau_\epsilon\big(\langle \delta_\lambda,(\delta_w \Gamma(e^{-\beta D}))(\delta_\lambda)\rangle\big)=\sum_{\lambda\in E^n}\tau_\epsilon\big(\langle \delta_\lambda,e^{-\beta n}\delta_w\cdot\delta_\lambda\rangle\big),
\]
where the dot in $\delta_w\cdot\delta_\lambda$ is the left action of $\delta_w \in C(E^0)$ on
$\delta_\lambda \in C(E^n)$. We can compute
\[
\langle \delta_\lambda,e^{-\beta n}\delta_w\cdot\delta_\lambda\rangle(v)
=\sum_{\nu\in E^nv}\overline{\delta_\lambda(\nu)}e^{-\beta n}(\delta_w\cdot\delta_\lambda)(\nu)=\begin{cases}
e^{-\beta n}\delta_w(r(\lambda))&\text{if $s(\lambda)=v$}\\
0&\text{otherwise,}
\end{cases}
\]
and therefore
\begin{align*}
(F^n\tau_\epsilon)(\delta_w)
    &=\sum_{\lambda\in E^n}\sum_{v\in E^0}\langle \delta_\lambda,e^{-\beta n}\delta_w\cdot\delta_\lambda\rangle(v)\epsilon_v\\
    &=\sum_{v\in E^0}\sum_{\lambda\in E^nv}e^{-\beta n}\delta_w(r(\lambda))\epsilon_v
    =\sum_{v\in E^0}e^{-\beta n}A^n(w,v)\epsilon_v.
\end{align*}
Thus the subinvariant vector $m=\sum_{n\geq 0}e^{-\beta n}A^n\epsilon=(I-e^{-\beta}A)^{-1}\epsilon$ in Theorem~\ref{mainthmk=1} is the trace $\sum_{n\geq 0}F^n\tau_\epsilon$ induced by the Fock module in the proof of \cite[Theorem~2.1]{LN}.

Laca and Neshveyev construct their KMS$_\beta$ state as follows. The trace $\tau_\epsilon$ on
$C(E^0)$ gives a finite-dimensional Hilbert space $H_\epsilon$ which carries a representation $M$
of $C(E^0)$ by multiplication operators. Rieffel induction then gives a representation
$\F(X)$-$\Ind M$ of $\LL(\F(X))$ on $\F(X)\otimes_{C(E^0)}H_\epsilon$. As above, the identity on
$X^{\otimes n}=C(E^n)$ is given by $1_n=\sum_{\lambda\in
E^n}\Theta_{\delta_\lambda,\delta_\lambda}$. Thus \cite[Theorem~1.1(ii)]{LN} implies that the trace
$\Tr_{\tau_\epsilon}$ satisfies
\[
\Tr_{\tau_\epsilon}(T)=\sum_{\lambda\in E^n}\tau_\epsilon\big(\langle\delta_\lambda,T\delta_\lambda\rangle\big)
=\sum_{\lambda\in E^n}\langle\delta_\lambda,T\delta_\lambda\rangle\epsilon_{s(\lambda)}.
\]
Thus the KMS$_\beta$ state $\phi$ in the proof of \cite[Theorem~2.1]{LN} is given by
\[
\phi(a)=\sum_{\lambda\in E^n}\langle\delta_\lambda,a\Gamma(e^{-\beta D})\delta_\lambda\rangle\epsilon_{s(\lambda)}
=\sum_{\lambda\in E^n}e^{-\beta n}\langle\delta_\lambda,a\delta_\lambda\rangle\epsilon_{s(\lambda)}.
\]
For $a=s_\mu s_\nu^*$, the inner product $\langle\delta_\lambda,a\delta_\lambda\rangle$ vanishes
unless $\mu=\nu$ and $\lambda=\mu\lambda'$, and then
$\langle\delta_\lambda,a\delta_\lambda\rangle=1$. Thus
\begin{align*}
\phi(s_\mu s_\mu^*)&=e^{-\beta|\mu|}\sum_{\lambda'\in s(\mu)E^*}e^{-\beta|\lambda'|}\epsilon_{s(\lambda')}
=e^{-\beta|\mu|}\sum_{n\geq 0}\big(e^{-\beta n}A^n\epsilon\big)_{s(\mu)}\\
&=e^{-\beta|\mu|}\big((1-e^{-\beta n}A^n)^{-1}\epsilon\big)_{s(\mu)}=e^{-\beta|\mu|}m_{s(\mu)}.
\end{align*}
Thus the state constructed in \cite[Theorem~2.1]{LN} is exactly the same as ours. In fact, the
representation $\F(X)$-$\Ind M$ is unitarily equivalent to the representation $\pi_{Q,T}$ on
$\ell^2(E^*)$ in the proof of Theorem~\ref{mainthmk=1} via the unitary which sends
$\delta_\mu\otimes \delta_{s(\mu)}$ to $\epsilon_{s(\mu)}^{1/2}h_\mu$.

\subsection{Groupoid algebras}
In \cite[Proposition~II.5.4]{Ren} Renault describes the KMS states for the $C^*$-algebras of
principal groupoids, which are the groupoids with no isotropy (see also \cite{KR}). Neshveyev has
recently extended Renault's result to groupoids with isotropy \cite[Theorem~1.3]{N}. Since graph
algebras have a groupoid model, Neshveyev's result applies in our situation. We will see that this
approach requires calculations similar to those in the proof of Theorem~\ref{mainthmk=1}, but they
may provide an instructive example for those interested in Neshveyev's approach. To avoid having
to adjust the standard groupoid model of \cite{KPRR} or \cite{Pat}, we assume that $E$ has no
sources.

The groupoid model $G$ for the Toeplitz algebra has unit space $G^{(0)}:=E^*\cup E^\infty$, and the
sets $Z(\alpha):=\{\alpha\beta:\beta\in G^{(0)}\}$ for $\alpha\in E^*$ form a basis of open compact
sets for the topology on $G^{(0)}$ (see \cite[Proposition~3.3]{Pat}). We have
\[
G=\big\{(\alpha y,|\gamma|-|\alpha|,\gamma y): y\in G^{(0)}, \alpha,\gamma\in E^*,  s(\alpha) = s(\gamma) = r(y)\big\},
\]
$r(y,k,z)=y$ and $s(y,k,z)=z$. The sets
\[
Z(\alpha,\gamma):=\{(\alpha y,|\gamma|-|\alpha|,\gamma y):y\in G^{(0)}, s(\alpha) = s(\gamma) = r(y)\}
\]
form a basis of compact-open sets for the topology on $G$. When we view $\T C^*(E)$ as $C^*(G)$, our dynamics is the one studied in \cite{N} for the cocycle $c:G\to\RR$ defined by $c(y,k,z)=k$. To construct KMS states on $C^*(G)$ using Neshveyev's theorem, we need to find quasi-invariant measures on $G^{(0)}$ and check some conditions involving the isotropy.

Fix $\beta>\ln\rho(A)$ and $\epsilon$ as in Theorem~\ref{mainthmk=1}\,\eqref{c}. Consider the
numbers $\{\Delta_\alpha:\alpha\in E^*\}$ constructed in the proof of that theorem. Since
$\sum_{\alpha\in E^*}\Delta_\alpha=1$, the discrete measure $\sum_{\alpha\in
E^*}\Delta_\alpha\delta_\alpha$ is a probability measure on $E^*$, and hence gives a probability
measure $\mu$ on $G^{(0)}$ such that
$\mu(Z(\alpha))=\sum_{r(\lambda)=s(\alpha)}\Delta_{\alpha\lambda}$ and $\mu(E^\infty)=0$. To see
that $\mu$ is quasi-invariant, consider the basic open set $U=Z(\alpha,s(\alpha))=\{(\alpha
y,-|\alpha|,y)\}$. Then $r$ and $s$ are homeomorphisms on $U$, and $T:=s|_U\circ(r|_U)^{-1}$ maps
$\alpha y$ to $y$. For $\lambda\in E^*$, we have
\[
T_*\mu(Z(\lambda))=\mu(T^{-1}(Z(\lambda)))
=\begin{cases}0&\text{if $r(\lambda)\not= s(\alpha)$}\\
\mu(Z(\alpha\lambda))&\text{if $r(\lambda)= s(\alpha)$.}\end{cases}
\]
Now the calculation
\begin{align*}
\mu(Z(\alpha\lambda))&=\sum_{r(\gamma)=s(\lambda)}\Delta_{\alpha\lambda\gamma}
=\sum_{r(\gamma)=s(\lambda)}e^{-\beta|\alpha\lambda\gamma|}\epsilon_{s(\gamma)}\\
&=e^{-\beta|\alpha|}\sum_{r(\gamma)=s(\lambda)}e^{-\beta|\lambda\gamma|}\epsilon_{s(\gamma)}=e^{-\beta|\alpha|}\mu(Z(\lambda))
\end{align*}
implies that $\frac{dT_*\mu}{d\mu}=e^{-\beta|\alpha|}$, and hence that $\mu$ is quasi-invariant with the correct cocycle (see the discussion preceding \cite[Theorem~1.3]{N}).

In the groupoid $G$, a unit $y$ has nonzero isotropy group $G_y^y=\{k:(y,k,y)\in G\}$ only if $y$ is an infinite path. Thus $\mu(\{y:G_y^y\not=\{0\}\})=0$, and the measurable fields of states described in conditions (ii) and (iii) of \cite[Theorem~1.3]{N} give no extra KMS states (see the comments following that theorem in \cite{N}). So Neshveyev's theorem gives a KMS$_\beta$ state $\psi_\epsilon$ on $\T C^*(E)$ which is the composition of the state $\mu_*$ on $C(G^{(0)})$ with the expectation $E$ of $C^*(G)$ onto $C(G^{(0)})$. We can check on elements of the form $\chi_{Z(\alpha)}=E(s_\alpha s_\alpha^*)$ that $\psi_\epsilon$ is the state $\phi_\epsilon$ in Theorem~\ref{mainthmk=1}.

We now suppose that $E$ is strongly connected, and consider states on $C^*(E)$. The groupoid model for $C^*(E)$ is the reduction of $G$ to $E^\infty$, and we can construct measures on $E^\infty$ by constructing functionals on $C(E^\infty)=\varinjlim C(E^n)$. Suppose that $\rho(A)>1$, and $x$ is the unimodular Perron-Frobenius eigenvector $A$. Then the relation $\mu_*(Z(\alpha))=\rho(A)^{-|\alpha|}x_{s(\alpha)}$ completely characterises a quasi-invariant measure $\mu$ on $E^\infty$. The paths $y$ with nonzero isotropy have the form $\alpha\gamma\gamma\gamma\cdots$ for some $\alpha,\gamma\in E^*$ with $s(\gamma)=r(\gamma)=s(\alpha)$; since $E$ is finite, there are countably many such paths, and since $\{\alpha\gamma\gamma\gamma\cdots\}=\bigcap_{n=1}^\infty Z(\alpha\gamma^n)$ we have $\mu(\{\alpha\gamma\gamma\gamma\cdots\})\leq \rho(A)^{-|\alpha|-n}$ for all $n$, and $\mu(\{\alpha\gamma\gamma\gamma\cdots\})=0$. Thus the units with nonzero isotropy have measure zero, and again the measurable fields of states in \cite[Theorem~1.3, (ii) and (iii)]{N} give no additional KMS states.

Now suppose that $\rho(A)=1$, in which case Lemma~\ref{valuesrho} implies that $E$ consists of a single cycle, say with $n$ edges. Then the path space $E^\infty$ has $n$ points, the
Perron-Frobenius eigenvector $x$ is constant, and $\mu$ is the uniformly distributed probability
measure. The isotropy group at $y\in E^\infty$ is $G_y^y=\{(y,nk,y):k\in \ZZ\}\cong n\ZZ$, so
Neshveyev's condition (iii) kicks in. His cocycle $c:(y,k,z)\mapsto k$ is injective on $G_y^y$, and
hence the only state $\phi$ of $C^*(G_y^y)=C^*(n\ZZ)$ such that $\phi(u_{nk})=0$ for all $k\in
c^{-1}(0)$ is evaluation at $0$ on $C_c(n\ZZ)\subset C^*(n\ZZ)$. So there is just one measurable
field satisfying (iii), and Neshveyev's theorem gives just one KMS$_0$ state on $C^*(E)$. This is
consistent, because he too requires his KMS$_0$ states to be invariant (just before
\cite[Theorem~1.3]{N}), and he makes it clear in the proof of the theorem that the purpose of (iii)
is to ensure invariance.

\appendix

\section{The spectral radius of a vertex matrix}

We now discuss some properties of the vertex matrices of finite graphs which we have used at various points. These are non-negative matrices in the sense of \cite[Chapter~1]{Seneta}, and have integer entries, so there must be a good chance that these results are known. But it is easy enough to give quick proofs.

\begin{lem}\label{valuesrho}
Suppose that $E$ is a finite graph, and let $A\in M_{E^0}(\NN)$ be the vertex matrix with entries $A(v,w)=|vE^1w|$.
\begin{enumerate}
\item\label{geq1cycle} If $E$ contains at least one cycle, then the spectral radius of $A$ satisfies $\rho(A)\geq 1$.
\item\label{nocycle} If $E$ contains no cycles, then $\rho(A)=0$.
\item\label{1cycle} If $E$ consists of disjoint cycles, then $\rho(A)=1$.
\item\label{notcycle} If $E$ is strongly connected and not a cycle, then $\rho(A)>1$.
\end{enumerate}
\end{lem}

\begin{proof}
\eqref{geq1cycle} For each $v\in E^0$ which lies on a cycle, we consider the set $C_v$ of vertices
$w$ which lie on a return path based at $v$. Two such $C_v$ and $C_u$ are either equal or disjoint;
the distinct subsets $C_1,\cdots,C_n$ are the classes (essential and inessential) which Seneta
discusses in \cite[Page~12]{Seneta}. By Seneta's structure theory, we can order the vertices so
that $A$ is a block upper triangular matrix whose square diagonal blocks are $1\times 1$ blocks of
the form $(0)$ and the irreducible submatrices $A_i$ of $A$ associated to the classes $C_i$. (Our matrix is upper rather than lower triangular because of the way we defined the
vertex matrix.) Since the determinant of a block-triangular matrix is the product of the
determinants of the diagonal blocks, the spectrum of $A$ is either $\bigcup_{i=1}^n\sigma(A_i)$ or
$\big(\bigcup_{i=1}^n\sigma(A_i)\big)\cup \{0\}$. So it suffices for us to prove that, if $A$ is
irreducible, as the $A_i$ are, then $\rho(A)\geq 1$.

By irreducibility, for each $v\in E^0$ there exists $m_v>0$ such that $A^{m_v}(v,v)>0$, and since
$A$ is an integer matrix, $A^{m_v}(v,v)\geq 1$. Now with $m:=\prod_{v\in E^0}m_v$ we have
$A^m(w,w)\geq 1$ for all $w\in E^0$. Thus the trace of $A^m$ satisfies $\Tr A^m\geq |E^0|$. Since
the trace is the sum of the complex eigenvalues counted according to multiplicity, at least one of
these eigenvalues has absolute value at least $1$, giving $\rho(A^m)\geq 1$. The spectral radius
formula implies that $\rho(A^m)=\rho(A)^m$, so we must have $\rho(A)\geq 1$ too.

\eqref{nocycle} In this case the structure theory of \cite[\S1.2]{Seneta} says that we can order the vertices so that $A$ is upper triangular with $0$s down the diagonal, and hence $\sigma(A)=\{0\}$.

\eqref{1cycle} In this case $A$ is a permutation matrix, and every eigenvalue is a root of unity.

\eqref{notcycle} Since $E$ is strongly connected, every vertex receives an edge. Since $E$ is not a
cycle, at least one vertex $v$ receives $2$ edges, $e$ and $f$, say. There are paths from $v$ to
$s(e)$ and $s(f)$, and hence there are cycles $\mu$ and $\nu$ based at $v$ such that $\mu_1=e$ and
$\nu_1=f$. Now $\mu^{|\nu|}$ and $\nu^{|\mu|}$ are distinct elements of $vE^{|\mu|\,|\nu|}v$, and
we have $A^{|\mu|\,|\nu|}(v,v)\geq 2$. With $m$ as in the proof of \eqref{geq1cycle},
$n:=|\mu|\,|\nu|m$ satisfies $A^n(w,w)\geq 1$ for all $w$ and $A^n(v,v)>1$, and the arguments at
the end of the proof of \eqref{geq1cycle} show that $\rho(A)>1$.
\end{proof}

\end{document}